\documentclass[12pt]{amsart}
\usepackage[top=1in, bottom=1in, left=1in, right=1in]{geometry}

\usepackage{amssymb,mathrsfs,amsmath}

\newtheorem{theorem}{Theorem}[section]
\newtheorem{lemma}[theorem]{Lemma}

\theoremstyle{remark}

\newcommand {\SN} {{\mathbb N}}
\newcommand {\SR} {{\mathbb R}}

\newcommand {\SC} {{\mathbb C}}
\newcommand {\SU} {{\mathbb U}}
\newcommand {\SD} {{\mathbb D}}

\newcommand{\bsp}{\begin{split}}
\newcommand{\esp}{\end{split}}
\newcommand{\be}{\begin{equation}}
\newcommand{\ee}{\end{equation}}
\newcommand{\bes}{\begin{equation*}}
\newcommand{\ees}{\end{equation*}}
\newcommand{\bv}\boldsymbol{}

\numberwithin{equation}{section}

\renewcommand{\pmod}[1]{~({\rm mod}\,#1)}

\begin{document}

\title[Pretentious multiplicative functions and primes in arithmetic progressions]{Pretentious multiplicative functions and the prime number theorem for arithmetic progressions}
\author{Dimitris Koukoulopoulos}
\address{D\'epartement de math\'ematiques et de statistique\\
Universit\'e de Montr\'eal\\
CP 6128 succ. Centre-Ville\\
Montr\'eal, QC H3C 3J7\\
Canada}
\email{koukoulo@dms.umontreal.ca}

\subjclass[2010]{11N05, 11N13, 11M06, 11M20}
\keywords{Prime number theorem for arithmetic progressions, multiplicative functions, Hal\'asz's theorem, elementary proofs of the prime number theorem, Siegel's theorem}

\date{\today}

\begin{abstract} Building on the concept of \textit{pretentious multiplicative functions}, we give a new and largely elementary proof of the best result known on the counting function of primes in arithmetic progressions.
\end{abstract}

\maketitle

\section{Introduction}\label{intro} In 1792 or 1793 Gauss conjectured that the number of primes up to $x$ is asymptotic to $x/\log x$, as $x\to\infty$, or, equivalently\footnote{It is customary to work with $\psi(x)$ instead of $\pi(x)=\sum_{p\le x}1$ for technical reasons. The two functions can be easily related using partial summation.}, that $$\psi(x):=\sum_{n\le x}\Lambda(n):=\sum_{\substack{p^k\le x\\p\ \text{prime},\,k\ge1}}\log p\sim x\quad(x\to\infty).$$ In 1896 this statement, now known as the Prime Number Theorem, was proven independently by de la Vall\'ee Poussin and Hadamard. Various authors improved upon this result and, currently, the best estimate known for $\psi(x)$, due to Korobov~\cite{korobov} and Vinogradov~\cite{vinogradov}, is
\be\label{pnt_e2}
\psi(x)=x+O\left(xe^{-c(\log x)^{3/5}(\log\log x)^{-1/5}}\right)\quad(x\ge3),
\ee
for some constant $c$. This result is far from what we expect to be the truth, since the Riemann Hypothesis is equivalent to the formula \be\label{rh}
\psi(x)=x+O(\sqrt{x}\log^2x)\quad(x\ge2)
\ee
(see \cite[p. 113]{Dav}).

The route to relation~\eqref{pnt_e2} is well known~\cite[Theorem 3.10]{titch}. Consider the Riemann $\zeta$ function, defined by $\zeta(s)=\sum_{n\ge1}1/n^s$ when $\Re(s)>1$. The better upper bounds we have available for $\zeta$ close to the line $\Re(s)=1$, the better estimates for $\psi(x)-x$ we can show. The innovation of Korobov and Vinogradov lies precisely on obtaining sharper bounds for $\zeta$, which can be reduced to estimating exponential sums of the form $\sum_{N<n\le2N}n^{it}$. However, the transition from these estimates to the error term in the prime number theorem uses heavily the analyticity of the Riemann $\zeta$ function.

For several years it was an open question whether it is possible to prove the prime number theorem circumventing the use of Complex Analysis. Philosophically speaking, prime numbers are defined in a very elementary way and there is no reason a priori why one should have to use such sophisticated machinery as the theory of complex analytic functions to understand their distribution. So it was a major breakthrough when Selberg~\cite{Sel1} and Erd\H os~\cite{Erd} succeeded in producing a completely elementary proof of the prime number theorem, in the sense that their arguments made no use of complex variables. Their method was subsequently improved significantly by Bombieri~\cite{bombieri} and Wirsing~\cite{wirsing}, and ultimately by Diamond and Steinig~\cite{DiamSt} who showed that 
$$
\psi(x)=x+O\left(xe^{-(\log x)^{1/7}(\log\log x)^{-2}} \right)
$$
Later, Daboussi~\cite{daboussi} gave a different completely elementary proof of the prime number theorem in the equivalent form
$$
\sum_{n\le x}\mu(n)=o(x)\quad(x\to\infty),
$$
where $\mu$ is the M\"obius function, defined to be $(-1)^{\#\{p|n\}}$ on squarefree integers $n$ and 0 otherwise. We refer the reader to~\cite[Section 2.4]{ik} and~\cite{goldfeld,granville} for further discussion about the elementary and the analytic proofs of the prime number theorem.

In the present paper we devise a new proof of~\eqref{pnt_e2} that uses only elementary estimates and Fourier inversion, thus closing the gap between complex analytic and more elementary methods. More generally, we give a new proof of the best estimate known for the number of primes in an arithmetic progression, which is Theorem~\ref{pntap} below\footnote{In fact, this theorem does not seem to have appeared in print before, even though it is well known to the experts. The needed zero-free region is mentioned in the notes of Chapter 9 of~\cite{mont2}; deducing Theorem~\ref{pntap} is then standard\ \cite[p. 120, 126]{Dav}. See also~\cite{LZ}.}. Our starting point is a proof by Iwaniec and Kowalski~\cite[p. 40-42]{ik} that the M\"obius function satisfies
\be\label{pnt ik}
\sum_{n\le x}\mu(n)\ll_A\frac x{(\log x)^A}\quad(x\ge2)
\ee
for any fixed positive $A$ (and hence that $\psi(x)=x+O_A(x/(\log x)^A)$ for all $A$). We improve upon this result by inserting into the proof some ideas from sieve methods together with estimates for exponential sums of the form $\sum_{N<n\le2N}(n+u)^{it}$, essentially due to van der Corput~\cite[Chapter 5]{titch} and Korobov-Vinogradov~\cite[Chapter 2]{walfisz}.

\begin{theorem}\label{pntap} Fix $A>0$. For $x\ge1$ and $(a,q)=1$ with $q\le(\log x)^A$ we have that
$$
\sum_{\substack{n\le x\\n\equiv a\pmod q}}\Lambda(n)=\frac x{\phi(q)}+O\left(xe^{-c_A(\log x)^{3/5}(\log\log x)^{-1/5}}\right)
$$
for some constant $c_A$, which cannot be computed effectively.
\end{theorem}

Theorem~\ref{pntap} will be proven in Section~\ref{proofs} as a corollary of the more general Theorem~\ref{pntap2}. As in the classical proof of Theorem~\ref{pntap}, in order to treat real Dirichlet characters and pass from Theorem~\ref{pntap2} to Theorem~\ref{pntap}, we need Theorem~\ref{siegel} below, which is due to Siegel~\cite{Si}. Pintz~\cite{Pi73} gave an elementary proof of this theorem, and in Section~\ref{siegelproof} we give an even simpler proof of it.

\begin{theorem}\label{siegel} Let $\epsilon\in(0,1]$. For all real, non-principal Dirichlet characters $\chi$ modulo $q$ we have that $L(1,\chi)\gg_\epsilon q^{-\epsilon}$; the implied constant cannot be computed effectively.
\end{theorem}

Lying just underneath the surface of the proof of Theorem~\ref{pntap} there is a general idea about multiplicative functions, which partly goes back to Hal\'asz~\cite{hal2,hal3}. Let $f:\SN\to\{z\in\SC:|z|\le1\}$ be a multiplicative function, namely $f$ satisfies the functional equation $$
f(mn)=f(m)f(n)\quad\text{whenever}\quad(m,n)=1.
$$
We want to understand when $f$ is small on average, that is to say, when
\be\label{small}
\sum_{n\le x}f(n)=o(x)\quad(x\to\infty).
\ee
First, note that if $f(n)=n^{it}$ for some fixed $t\in\SR$, then~\eqref{small} does not hold. Also, if we tweak $n^{it}$ a little bit, we find more counterexamples to~\eqref{small}. Hal\'asz showed that these are the only counterexamples: unless $f$ \textit{pretends to be} $n^{it}$ for some $t\in\SR$, in the sense that
$$
\sum_p\frac{1-\Re(f(p)p^{-it})}p<\infty,
$$
then relation~\eqref{small} holds.

Letting $f=\mu$ in Hal\'asz's theorem, we find that the prime number theorem is reduced to the statement that $\mu$ does not pretend to be $n^{it}$ for any $t$. Recently, Granville and Soundararajan~\cite{gs} carried out this argument and obtained a new proof of the prime number theorem. The possibility that $\mu$ pretends to be $n^{it}$ was excluded using Selberg's sieve. From a broad point of view, this proof can be also regarded elementary, albeit not completely elementary, since implicit in the proof of Hal\'asz's theorem is Plancherel's formula from Harmonic Analysis. Using a quantitative version of Hal\'asz's theorem, due to various authors~\cite{mont,ten,gs1}, this proof produces the estimate
\be\label{gs}
\sum_{n\le x}\mu(n)\ll\frac x{(\log x)^{1-2/\pi+o(1)}}\quad(x\to\infty).
\ee
The weak error term in~\eqref{gs} is intrinsic to the method of Graville and Soundararajan. The generality of Hal\'asz's theorem is simultaneously its Achilles' heel: it can never yield an error term that is better than $x\log\log x/\log x$ (\cite{mv2,gs}).

These quantitative limitations of Hal\'asz's theorem were the stumbling block to obtaining good bounds on the error term in the prime number theorem using general tools about multiplicative functions. In~\cite{kou} we develop further the methods of this paper and show how to obtain an improvement over Hal\'asz's theorem for a certain class of functions $f$: we show that if
$$
\sum_{\substack{n\le x \\ (n,2)=1}} f(n)\ll_A\frac x{(\log x)^A}\quad(x\ge2)
$$
for some $A\ge6$, then we have sharper bounds for the partial sums of $\mu\cdot f$ than what Hal\'asz's theorem gives us. In particular, if $f$ does not pretend to be $\mu(n)n^{it}$ for some $t\in\SR$, then we can show a lot of cancelation in the partial sums of $\mu\cdot f$. The key observation for the application to primes in arithmetic progressions is that a non-principal character $\chi$ is very small on average. So using our methods we may show that $\mu\cdot\chi$ is also very small on average, provided that $\chi$ does not pretend to be $\mu(n)n^{it}$ for some $t\in\SR$, which, in more classical terms, corresponds to a suitable zero-free region for the $L$-function $L(s,\chi)=\sum_{n\ge1}\chi(n)/n^s$ around the point $1+it$.

%%%%%%%%%%%%%%%%%%%%%%%%%%%%%%%%%%%%%%%%%%%%%%%%%%%%%%%%%%%%%%%%%%%%%%%%%%%%%%%%%%%%%%%%%%%%%%%%%%%%%%%%%%%%%%%%%%%%%%%%%%%%
%%%%%%%%%%%%%%%%%%%%%%%%%%%%%%%%%%%%%%%%%%%%%%%%%%%%%%%%%%%%%%%%%%%%%%%%%%%%%%%%%%%%%%%%%%%%%%%%%%%%%%%%%%%%%%%%%%%%%%%%%%%%
%%%%%%%%%%%%%%%%%%%%%%%%%%%%%%%%%%%%%%%%%%%%%%%%%%%%%%%%%%%%%%%%%%%%%%%%%%%%%%%%%%%%%%%%%%%%%%%%%%%%%%%%%%%%%%%%%%%%%%%%%%%%
%%%%%%%%%%%%%%%%%%%%%%%%%%%%%%%%%%%%%%%%%%%%%%%%%%%%%%%%%%%%%%%%%%%%%%%%%%%%%%%%%%%%%%%%%%%%%%%%%%%%%%%%%%%%%%%%%%%%%%%%%%%%
%%%%%%%%%%%%%%%%%%%%%%%%%%%%%%%%%%%%%%%%%%%%%%%%%%%%%%%%%%%%%%%%%%%%%%%%%%%%%%%%%%%%%%%%%%%%%%%%%%%%%%%%%%%%%%%%%%%%%%%%%%%%

\section{Preliminaries}

\subsection*{Notation} For an integer $n$ we denote with $P^+(n)$ and $P^-(n)$ the greatest and smallest prime divisors of $n$, respectively, with the notational convention that $P^+(1)=1$ and $P^-(1)=\infty$, and we let $\tau_r(n)=\sum_{d_1\cdots d_r=n}1$. Given two arithmetics functions $f,g:\SN\to\SC$, we write $f*g$ for their Dirichlet convolution, defined by $(f*g)(n)=\sum_{ab=n}f(a)g(b)$. The notation $F\ll_{a,b,\dots}G$ means that $|F|\le CG$, where $C$ is a constant that depends at most on the subscripts $a,b,\dots$, and $F\asymp_{a,b,\dots} G$ means that $F\ll_{a,b,\dots}G$ and $G\ll_{a,b,\dots}F$. Finally, we use the letter $c$ to denote a constant, not necessarily the same one in every place, and possibly depending on certain parameters that will be specified using subscripts and other means.

\medskip

%%%%%%%%%%%%%%%%%%%%%%%%%%%%%%%%%%%%%%%%%%%%%%%%%%%%%%%%%%%%%%%%%%%%%%%%%%%%%%%%%%%%%%%%%%%%%%%%%%%%%%%%%%%%%%%%%%%%%%%%%%%%
%%%%%%%%%%%%%%%%%%%%%%%%%%%%%%%%%%%%%%%%%%%%%%%%%%%%%%%%%%%%%%%%%%%%%%%%%%%%%%%%%%%%%%%%%%%%%%%%%%%%%%%%%%%%%%%%%%%%%%%%%%%%
%%%%%%%%%%%%%%%%%%%%%%%%%%%%%%%%%%%%%%%%%%%%%%%%%%%%%%%%%%%%%%%%%%%%%%%%%%%%%%%%%%%%%%%%%%%%%%%%%%%%%%%%%%%%%%%%%%%%%%%%%%%%
%%%%%%%%%%%%%%%%%%%%%%%%%%%%%%%%%%%%%%%%%%%%%%%%%%%%%%%%%%%%%%%%%%%%%%%%%%%%%%%%%%%%%%%%%%%%%%%%%%%%%%%%%%%%%%%%%%%%%%%%%%%%
%%%%%%%%%%%%%%%%%%%%%%%%%%%%%%%%%%%%%%%%%%%%%%%%%%%%%%%%%%%%%%%%%%%%%%%%%%%%%%%%%%%%%%%%%%%%%%%%%%%%%%%%%%%%%%%%%%%%%%%%%%%%

In this section we present a series of auxiliary results we will need later. We start with the following lemma, which is based on an idea in~\cite[p. 40]{ik}.

\begin{lemma}\label{derlemma} Let $M\ge1$, $D$ be an open subset of $\SC$ and $s\in D$. Consider a function $F:D\to\SC$ that is differentiable $k$ times at $s$ and its derivatives satisfy the bound $|F^{(j)}(s)|\le j!M^j$ for $1\le j\le k$. If $F(s)\neq0$, then $$\left|\left(\frac{F'}{F}\right)^{(k-1)}(s)\right|\le \frac{k!}2\left(\frac{2M}{\min\{|F(s)|,1\}}\right)^k.$$
\end{lemma}

\begin{proof} We have the identity
\bes\label{id3}
\left(\frac{-F'}{F}\right)^{(k-1)}(s)=k!\sum_{a_1+2a_2+\cdots=k}\frac{(-1+a_1+a_2+\cdots)!}{a_1!a_2!\cdots}
\left(\frac{-F'}{1!F}(s)\right)^{a_1}\left(\frac{-F''}{2!F}(s)\right)^{a_2}\cdots,
\ees
which can be easily verified by induction\footnote{A similar identity for the derivatives of $1/F$ was used by Iwaniec and Kowalski~\cite[p. 40]{ik} in their proof of~\eqref{pnt ik}.}. In order to complete the proof of the lemma, we will show that
\be\label{pr e0}
\sum_{a_1+2a_2+\cdots+ka_k=k}\frac{(a_1+a_2+\cdots+a_k)!}{a_1!a_2!\cdots a_k!}=\sum_{a_1+2a_2+\cdots+ka_k=k}\binom{a_1+a_2+\cdots+a_k}{a_1,a_2,\dots,a_k}=2^{k-1}.
\ee
Indeed, for each fixed $k$-tuple $(a_1,\dots,a_k)\in(\SN\cup\{0\})^k$ with $a_1+2a_2+\cdots+ka_k=k$, the multinomial coefficient $\binom{a_1+a_2+\cdots+a_k}{a_1,a_2,\dots,a_k}$ represents the way of writing $k$ as the sum of $a_1$ ones, $a_2$ twos, and so on, with the order of the different summands being important, e.g. if $k=5$, $a_1=1$, $a_2=2$ and $a_3=a_4=a_5=0$, then there are three such ways to write 5: $5=2+2+1=2+1+2=1+2+2$. So we conclude that
\bes
\sum_{a_1+2a_2+\cdots+ka_k=k}\binom{a_1+a_2+\cdots+a_k}{a_1,a_2,\dots,a_k}=\#\{\text{ordered partitions of}\ k\},
\ees
where we define an ordered partition of $k$ to be a way to write $k$ as the sum of positive integers, with the order of the different summands being important. To every ordered partition of $k=b_1+\cdots+b_m$, we can associate a unique subset of $\{1,\dots,k\}$ in the following way: consider the set $B\subset\{1,\dots,k\}$ which contains $\{1,\dots,b_1\}$, does not contain $\{b_1+1,\dots,b_1+b_2\}$, contains $\{b_1+b_2+1,\dots,b_1+b_2+b_3\}$, does not contain $\{b_1+b_2+b_3+1,\dots,b_1+b_2+b_3+b_4\}$, and so on. Then $B$ necessarily contains $1$ and, conversely, every subset of $\{1,\dots,k\}$ containing 1 can arise this way. So we conclude that there are $2^{k-1}$ ordered partitions, and~\eqref{pr e0} follows, thus completing the proof of the lemma.
\end{proof}

Below we state a result which is known as the {\it fundamental lemma of sieve methods}. It has appeared in the literature in many different forms (for example, see~\cite[Theorem 7.2]{halb}). The version we shall use is Lemma 5 in~\cite{fi}.

\begin{lemma}\label{fund} Let $y\ge2$ and $D=y^u$ with $u\ge2$. There exist two arithmetic functions $\lambda^\pm:\SN\to[-1,1]$ supported in $\{d\in\SN\cap[1,D]:P^+(d)\le y\}$ and such that
$$
\begin{cases}(\lambda^-*1)(n)=(\lambda^+*1)(n)=1&\text{if}~P^-(n)>y,\cr (\lambda^-*1)(n)\le0\le(\lambda^+*1)(n)&\text{otherwise}.\end{cases}
$$
Moreover, if $g:\SN\to[0,1]$ is a multiplicative function, and either $\lambda=\lambda^+$ or $\lambda=\lambda^-$, then
$$
\sum_d\frac{\lambda(d)g(d)}d=\left(1+O(e^{-u})\right)\prod_{p\le y}\left(1-\frac{g(p)}p\right).
$$
\end{lemma}

In addition, we need estimates for the exponential sums $\sum_{N<n\le2N}(n+u)^{it}$ when $\log N\ll\log|t|$. The best such result that is known is due to Ford~\cite[Theorem 2]{Fo}, but its proof uses implicitly the zeroes of the Riemann zeta function (see Lemmas 2.1 and 2.2 in~\cite{Fo}). Instead, we use results from~\cite[Chapter 5]{titch} and~\cite[Chapter 2]{walfisz}\ that are due to van der Corput and Korobov-Vinogradov, respectively, and have completely elementary proofs.

\begin{lemma}\label{nit} For $t\in\SR$, $0\le u\le1$ and $2\le N\le t^2$ we have that 
$$
\sum_{N<n\le2N}(n+u)^{it}\ll N\cdot\exp\left\{-\frac{(\log N)^3}{66852(\log|t|)^2}\right\}.
$$
\end{lemma}

\begin{proof} When $N\le|t|^{1/18}$, the result follows by~\cite[Satz 2, p. 57]{walfisz}, since 
$$
\left\lfloor\frac{\log|t|}{\log N}\right\rfloor+1\le\frac{19}{18}\frac{\log|t|}{\log N}
$$ 
for such $N$ and $66852\ge(19/18)^260000$. Assume now that $|t|^{1/18}<N\le|t|^2$. Consider $m\in\{1,2,\dots,35\}$ such that $|t|^{\frac2{m+1}}<N\le |t|^{\frac2m}$. Let $k=\left\lfloor \frac m2+\frac74\right\rfloor\ge2$, so that $N^{k-7/4}\le|t|\le N^{k-1/4}$, and set $K=2^{k-1}\ge2$. Then we apply Theorem 5.9, 5.11 or 5.13 in~\cite[Section 5.9, p. 104-107]{titch} according to whether $k=2$, $k=3$ or $k\ge4$, respectively, to obtain the estimate
\bes\bsp
\frac1N\sum_{N<n\le2N}(n+u)^{it}\ll\left(\frac{|t|}{N^k}\right)^{\frac1{2K-2}}+\left(\frac{N^{k-2}}{|t|}\right)^{\frac1{2K-2}}\ll \frac1{N^{\frac1{8K}}}\le\exp\left\{-\frac{\left(k-\frac74\right)^2}{2^{k+2}}\frac{(\log N)^3}{(\log|t|)^2}\right\}.
\end{split}\ees
Since $2^{k+2}/(k-7/4)^2\le7050$ for $2\le k\le19$, the lemma follows in this last case too.
\end{proof}

Next, we show how to combine Lemmas~\ref{fund} and~\ref{nit} to estimate sums of $\chi(n)n^{it}$ when $n$ has no small prime factors. Here and for the rest of the paper we use the notation 
$$
\delta(\chi)
	=\begin{cases}
		1&\text{if}\ \chi\ \text{is a principal Dirichlet character},\cr 
		0&\text{otherwise},
	\end{cases}
$$
and 
$$
V_t=\exp\left\{(\log(3+|t|))^{2/3}(\log\log(3+|t|))^{1/3}\right\}\quad(t\in\SR).
$$

\begin{lemma}\label{chinit} Let $\chi$ be a Dirichlet character modulo $q$ and $t\in\SR$. For $x\ge y\ge 2$ with $x\ge \max\{q^4,V_t^{100}\}$ we have that 
\bes\bsp\sum_{\substack{n\le x\\P^-(n)>y}}\chi(n)n^{it}
	&=\frac{\delta(\chi)\phi(q)}q\frac{x^{1+it}}{1+it}\prod_{\substack{p\le y\\p\nmid q}}\left(1-\frac1p\right)	
	+O\left(\frac{x^{1-1/(30\log y)} + x^{1-1/(100\log V_t)} }{\log y}\right).
\end{split}\ees
\end{lemma}

\begin{proof} We apply Lemma~\ref{fund} with $D=(x/q)^{1/20}\ge x^{3/80}$. We have that
\be\label{pre1}\bsp
\sum_{\substack{n\le x\\P^-(n)>y}}\chi(n)n^{it}&=\sum_{n\le x}(\lambda^+*1)(n)\chi(n)n^{it}+O\left(\sum_{n\le x}(\lambda^+*1-\lambda^-*1)(n)\right)\\
&=\sum_d\lambda^+(d)\chi(d)d^{it}\sum_{m\le x/d}\chi(m)m^{it}+O\left(\sum_d(\lambda^+(d)-\lambda^-(d))\left\lfloor\frac xd\right\rfloor\right).
\end{split}\ee
For the error term, note that \be\label{pre3}\sum_d(\lambda^+(d)-\lambda^-(d))\left\lfloor\frac xd\right\rfloor=x\sum_d\frac{\lambda^+(d)-\lambda^-(d)}d+O\left((x/q)^{1/20}\right)\ll\frac{x^{1-3/(80\log y)}}{\log y},\ee by Lemma~\ref{fund}. To estimate the main term of the right hand side of~\eqref{pre1}, we distinguish two cases.

First, assume that $x>qt^2/c$, where $c$ is some large constant to be chosen later. Then
\bes\label{pre2}\bsp
\sum_{m\le x/d}\chi(m)m^{it} =\int_{1^-}^{x/d}u^{it}d\left(\frac{\delta(\chi)\phi(q)}qu+O(q)\right) &=\frac{\delta(\chi)\phi(q)}q\frac{(x/d)^{1+it}}{1+it}+O(q(1+|t|)\log x)\\
&=\frac{\delta(\chi)\phi(q)}q\frac{(x/d)^{1+it}}{1+it}+O_c(x^{5/8}\log x),
\end{split}\ees
since $q(|t|+1)\ll_c\sqrt{xq}\le x^{5/8}$. Inserting the above estimate and~\eqref{pre3} into~\eqref{pre1} yields that
\be\label{pre10}\bsp
\sum_{\substack{n\le x\\P^-(n)>y}}\chi(n)n^{it} &= \frac{\delta(\chi)\phi(q)}{q} \frac{x^{1+it}}{1+it} \sum_d\frac{\lambda^+(d)\chi(d)}{d}
+  O_c\left( x^{5/8+1/20}\log x+ \frac{ x^{1-3/(80\log y)} } {\log y} \right) .
\end{split}\ee
If $\delta(\chi)=0$, the first term on the right hand side of~\eqref{pre10} vanishes trivially. Otherwise, $\chi$ is principal and, consequently,
\bes\bsp
\sum_d\frac{\lambda^+(d)\chi(d)}{d} = \sum_{(d,q)=1}\frac{\lambda^+(d)}d
&= \left(1+O\left(x^{-3/(80\log y)}\right)\right) \prod_{\substack{p\le y\\p\nmid q}}\left(1-\frac1p\right) \\
& = \prod_{ \substack{p\le y\\p\nmid q} }\left(1-\frac1p\right)+O\left( \frac{q}{\phi(q)} \frac{x^{-3/(80\log y)} } {\log y} \right),
\end{split}\ees
by Lemma~\ref{fund}. In any case, we have that
\bes
\sum_{\substack{n\le x\\P^-(n)>y}}\chi(n)n^{it} = \frac{\delta(\chi)\phi(q)}{q} \frac{x^{1+it}}{1+it} \prod_{ \substack{p\le y\\p\nmid q} }\left(1-\frac1p\right)
+  O_c\left(\frac{ x^{1-3/(80\log y)} } {\log y} \right),
\ees
which completes the proof in this case.

Finally, assume that $x\le qt^2/c$. In this case we must have that $|t|\ge\sqrt{c}$, since $x\ge q^4$. Moreover, observe that, for $2\le N\le t^2$ and $u\in[0,1]$, Lemma~\ref{nit} implies that
\begin{align*}
\sum_{n\le N-u}(n+u)^{it} &= O(\sqrt{N})+\sum_{1\le2^j\le\sqrt{N}} \sum_{\frac{N-u}{2^{j+1}}<n\le\frac{N-u}{2^j} }(n+u)^{it} \\
&\ll \sqrt{N} +\sum_{1\le2^j\le\sqrt{N}} \frac{N}{2^j} \cdot \exp\left\{ -  \frac{(\log(N-u)-(j+1)\log 2)^3}{66852(\log|t|)^2} \right\} \\
&\ll \sqrt{N} + N \cdot \exp\left\{ -  \frac{\left(\log\frac{N-u}2\right)^3}{66852(\log|t|)^2} \right\} \ll N \cdot \exp\left\{ -  \frac{(\log N)^3}{66852(\log|t|)^2} \right\}
\end{align*}
because the summands on the second line of the above relation decay exponentially in $j$ by our assumption that $N\le t^2$. So for $d\le(x/q)^{1/20}$ we have that
\bes\bsp
\sum_{n\le x/d}\chi(n)n^{it} &=\sum_{a=1}^{q}\chi(a)\sum_{\substack{n\le x/d\\n\equiv\,a\,(\text{mod}\ q)}}n^{it}
=\sum_{a=1}^q\chi(a)q^{it}\sum_{0\le j\le\frac{x/d-a}q}\left(j+\frac aq\right)^{it} \\
&\ll\sum_{a=1}^q \frac x{dq}\exp\left\{-\frac{\log^3(x/(dq))}{66852(\log|t|)^2}\right\} \le \frac xd \exp\left\{-\frac{(\log x)^3}{185000(\log |t|)^2}\right\},\end{split}\ees since $x/(dq)\ge(x/q)^{19/20}\ge x^{0.7125}$ for $d\le(x/q)^{20}$ and $q\le x^{1/4}$. Inserting this estimate and~\eqref{pre3} into~\eqref{pre1}, we deduce that
\be\label{pre4}
\sum_{\substack{n\le x\\P^-(n)>y}}\chi(n)n^{it}\ll x(\log y)\exp\left\{-\frac{(\log x)^3}{185000(\log |t|)^2}\right\} + \frac{x^{1-3/(80\log y)}}{\log y}.
\ee
Lastly, observe that our assumption that $x\ge V_t^{100}$ implies that
\begin{align*}
\frac{\log x}{(\log\log x)^{1/3}(\log |t|)^{2/3}}\ge\frac{100\log V_t}{(\log(100\log V_t))^{1/3}(\log |t|)^{2/3}} \sim100\left(\frac32\right)^{1/3},
\end{align*} as $|t|\to\infty$. Since $|t|\ge\sqrt{c}$, taking $c$ large enough we find that
$\log x\ge100(\log\log x)^{1/3}(\log|t|)^{2/3}$ and, consequently,
\[
\frac{ (\log x)^3} { 185000 (\log |t|)^2}
    = \frac{ (\log x)^3} { 370000 (\log |t|)^2}
        + \frac{ (\log x)^3} { 370000 (\log |t|)^2}
    \ge \left( \frac{ \log x }  { 100\log V_t }\right)^3     + 2\log\log x
\]
Inserting the above inequality into~\eqref{pre4} completes the proof of the lemma in this last case too (note that the main term in the statement of the lemma is smaller than the error term in this case, since $|t|\ge\sqrt{cx/q}\ge c^{1/2}x^{3/8}$ by assumption).
\end{proof}

%%%%%%%%%%%%%%%%%%%%%%%%%%%%%%%%%%%%%%%%%%%%%%%%%%%%%%%%%%%%%%%%%%%%%%%%%%%%%%%%%%%%%%%%%%%%%%%%%%%%%%%%%%%%%%%%%%%%%%%%%%%%
%%%%%%%%%%%%%%%%%%%%%%%%%%%%%%%%%%%%%%%%%%%%%%%%%%%%%%%%%%%%%%%%%%%%%%%%%%%%%%%%%%%%%%%%%%%%%%%%%%%%%%%%%%%%%%%%%%%%%%%%%%%%
%%%%%%%%%%%%%%%%%%%%%%%%%%%%%%%%%%%%%%%%%%%%%%%%%%%%%%%%%%%%%%%%%%%%%%%%%%%%%%%%%%%%%%%%%%%%%%%%%%%%%%%%%%%%%%%%%%%%%%%%%%%%
%%%%%%%%%%%%%%%%%%%%%%%%%%%%%%%%%%%%%%%%%%%%%%%%%%%%%%%%%%%%%%%%%%%%%%%%%%%%%%%%%%%%%%%%%%%%%%%%%%%%%%%%%%%%%%%%%%%%%%%%%%%%
%%%%%%%%%%%%%%%%%%%%%%%%%%%%%%%%%%%%%%%%%%%%%%%%%%%%%%%%%%%%%%%%%%%%%%%%%%%%%%%%%%%%%%%%%%%%%%%%%%%%%%%%%%%%%%%%%%%%%%%%%%%%

\section{Distances of multiplicative functions}\label{dist}

Given two multiplicative functions $f,g:\SN\to\{z\in\SC:|z|\le1\}$ and real numbers $x\ge y\ge1$, we set $$\SD(f,g;y,x)=\left(\sum_{y<p\le x}\frac{1-\Re(f(p)\overline{g(p)})}p\right)^{1/2}.$$ This quantity defines a certain measure of `distance' between $f$ and $g$. In particular, it satisfies the following triangle-like inequality~\cite[p. 207]{gs2}.

\begin{lemma}\label{triangle}Let $f,g:\SN\to\{z\in\SC:|z|\le1\}$ be multiplicative functions and $x\ge y\ge1$. Then 
$$
\SD(1,f;y,x)+\SD(1,g;y,x)\ge\SD(1,fg;y,x).
$$
\end{lemma}

It is possible to relate the distance of two multiplicative functions to the value of a certain Dirichlet series. Here and for the rest of the paper, given an arithmetic function $f:\SN\to\SR$, we let $$L(s,f)=\sum_{n=1}^\infty\frac{f(n)}{n^s}\quad\text{and}\quad L_y(s,f)=\sum_{P^-(n)>y}\frac{f(n)}{n^s},$$ provided the series converge.

\begin{lemma}\label{l1} Let $x,y\ge2$, $t\in\SR$ and $f:\SN\to\{z\in\SC:|z|\le1\}$ be a multiplicative function. Then 
$$
\log\left|L_y\left(1+\frac1{\log x}+it,f\right)\right|=\sum_{y<p\le x}\frac{\Re(f(p)p^{-it})}{p}+O(1).
$$
\end{lemma}

\begin{proof} This result has appeared in the literature before, at least implicitly (see \cite[Lemma 6.6 in p. 230, and p. 253]{el1} as well as~\cite[p. 459-460]{ten}). The current formulation is due to Granville and Soundararajan~\cite{gs}. The proof follows by writing $L_y(1+1/\log x+it,f)$ as an Euler product and observing that\footnote{Note that $|f(p)/p^s+f(p^2)/p^{2s}+\cdots|\le1/(p^\sigma-1)\le1/2$ for $p>y\ge2$ and $\Re(s)>1$.}
$\log|1+z|=\Re(z)+O(|z|^2)$ for $|z|\le1/2$, and that
\be\label{dist e2}
\sum_{p\le x}\left(\frac1p-\frac1{p^{1+1/\log x}}\right)+ \sum_{p>x} \frac1{p^{1+1/\log x}} \ll1,
\ee
by Chebyshev's estimate $\sum_{p\le u}\log p\ll u$ and partial summation.
\end{proof}

We conclude this section with the following lemma, which is a partial demonstration of the phenomenon mentioned towards the end of the introduction: if a multiplicative function $f:\SN\to\SU$ is small on average and it does not pretend to be $\mu(n)n^{it}$ for some $t$, then $f(p)$ is small on average.

\begin{lemma}\label{dist1} Let $y_2\ge y_1\ge y_0\ge2$. Consider a multiplicative function $f:\SN\to\{z\in\SC:|z|\le1\}$ such that $$\left|L_{y_0}'\left(1+\frac1{\log x},f\right)\right|\le c\log y_0\quad(y_1\le x\le y_2)$$ for some $c\ge1$ and $$\SD^2(f(n),\mu(n);y_0,x)\ge\delta\log\left(\frac{\log x}{\log y_0}\right)-M\quad(y_1\le x\le y_2)$$ for some $\delta>0$ and $M\ge0$. Then $$\left|\sum_{y_1<p\le y_2}\frac{f(p)}p\right|\ll_{c,M}\frac1\delta.$$
\end{lemma}

\begin{proof} Let $s=\sigma+it\in\SC$ with $\sigma>1$ and $t\in\SR$. Note that $|f(p)/p^s+f(p^2)/p^{2s}+\cdots|\le1/(p^\sigma-1)\le1/2$ for $p>y_0\ge2$, so that
$$
\left(1+\frac{f(p)}{p^s}+\frac{f(p^2)}{p^{2s}}+\cdots\right)^{-1}=1+O\left(\frac1{p^\sigma}\right)\quad(p>y_0\ge2).
$$
Now, since $L_{y_0}(s,f)=\prod_{p>y_0}(1+f(p)/p^s+f(p^2)/p^{2s}+\cdots)$, logarithmic differentiation yields the formula
\begin{align*}
-\frac{L_{y_0}'}{L_{y_0}}(s,f)
&=\sum_{p>y_0} \frac{f(p)(\log p)/p^s+2f(p^2)(\log p)/p^{2s}+\cdots}{1+f(p)/p^s+f(p^2)/p^{2s}+\cdots}\\
&=\sum_{p>y_0}\left(\frac{f(p)\log p}{p^s}+O\left(\frac{\log p}{p^{2\sigma}}\right)\right)\left(1+O\left(\frac1{p^\sigma}\right)\right)\\
&=\sum_{p>y_0}\left(\frac{f(p)\log p}{p^s}+O\left(\frac{\log p}{p^{2\sigma}}\right)\right)
=\sum_{p>y_0}\frac{f(p)\log p}{p^s}+O(1).
\end{align*}
Moreover, relation~\eqref{dist e2} implies that
$$
\sum_{y_0<p\le z}\frac{f(p)}p=\sum_{p>y_0}\frac{f(p)}{p^{1+1/\log z}}+O(1)\quad(z\ge y_0).
$$
Combining the above formulas, we find that
\be\label{dist1 e10}\bsp
\sum_{y_1<p\le y_2}\frac{f(p)}p&=O(1)+\sum_{p>y_0}\frac{f(p)}{p^{1+1/\log y_2}}-\sum_{p>y_0}\frac{f(p)}{p^{1+1/\log y_1}}\\
&=O(1) + \sum_{p>y_0} \frac{f(p)\log p}{p} \int_{y_1}^{y_2}\frac{du}{p^{1/\log u}u\log^2u}\\
&=O(1) + \int_{y_1}^{y_2} \sum_{p>y_0} \frac{f(p)\log p}{p^{1+1/\log u}}\frac{du}{u\log^2u}\\
&=O(1) - \int_{y_1}^{y_2} \frac{L_{y_0}'}{L_{y_0}}\left(1+\frac1{\log u},f\right)\frac{du}{u\log^2u}.
\end{split}\ee
Moreover, Lemma~\ref{l1} and our assumptions on $f$ imply that, for $u\in[y_1,y_2]$, we have that
$$
\left|L_{y_0}\left(1+\frac1{\log u},f\right)\right|\asymp\left(\frac{\log y_0}{\log u}\right)\exp\left\{\SD^2(f(n),\mu(n);y_0,u)\right\}\gg_M\left(\frac{\log u}{\log y_0}\right)^{-1+\delta}
$$
and
$|L_{y_0}'(1+1/\log u,f)|\le c\log y_0$. Inserting these estimates into~\eqref{dist1 e10}, we conclude that
$$
\sum_{y_1<p\le y_2}\frac{f(p)}p\ll_{c,M}1+\int_{y_1}^{y_2}(\log y_0)\left(\frac{\log u}{\log y_0}\right)^{1-\delta} \frac{du}{u(\log u)^2}
=1+(\log y_0)^\delta\int_{y_1}^{y_2}\frac{du}{u(\log u)^{1+\delta}} \ll\frac1\delta,
$$
thus completing the proof of the lemma.

\end{proof}

%%%%%%%%%%%%%%%%%%%%%%%%%%%%%%%%%%%%%%%%%%%%%%%%%%%%%%%%%%%%%%%%%%%%%%%%%%%%%%%%%%%%%%%%%%%%%%%%%%%%%%%%%%%%%%%%%%%%%%%%%%%%
%%%%%%%%%%%%%%%%%%%%%%%%%%%%%%%%%%%%%%%%%%%%%%%%%%%%%%%%%%%%%%%%%%%%%%%%%%%%%%%%%%%%%%%%%%%%%%%%%%%%%%%%%%%%%%%%%%%%%%%%%%%%
%%%%%%%%%%%%%%%%%%%%%%%%%%%%%%%%%%%%%%%%%%%%%%%%%%%%%%%%%%%%%%%%%%%%%%%%%%%%%%%%%%%%%%%%%%%%%%%%%%%%%%%%%%%%%%%%%%%%%%%%%%%%
%%%%%%%%%%%%%%%%%%%%%%%%%%%%%%%%%%%%%%%%%%%%%%%%%%%%%%%%%%%%%%%%%%%%%%%%%%%%%%%%%%%%%%%%%%%%%%%%%%%%%%%%%%%%%%%%%%%%%%%%%%%%
%%%%%%%%%%%%%%%%%%%%%%%%%%%%%%%%%%%%%%%%%%%%%%%%%%%%%%%%%%%%%%%%%%%%%%%%%%%%%%%%%%%%%%%%%%%%%%%%%%%%%%%%%%%%%%%%%%%%%%%%%%%%

\section{Bounds on Dirichlet $L$-functions}\label{lchi}

In this section we list some essential estimates on the derivatives of $L_y(s,\chi)$ and $\frac{L_y'}{L_y}(s,\chi)$. We start with the following result, which is a consequence of Lemma~\ref{chinit}.

\begin{lemma}\label{lchil1} Let $\chi$ be a Dirichlet character modulo $q$, $k\in\SN\cup\{0\}$, and $s=\sigma+it$ with $\sigma>1$ and $t\in\SR$. For $y\ge3/2$ we have that
\be\label{lchil1 e1}
\left|L_y^{(k)}(s,\chi)+\frac{(-1)^{k+1}k!}{(s-1)^{k+1}}\frac{\delta(\chi)\phi(q)}q\prod_{\substack{p\le y\\p\nmid q}}\left(1-\frac1p\right)\right|\ll\frac{k!(c\log(yqV_t))^{k+1}}{\log y}.
\ee
In particular, if $y\ge\max\left\{qV_t,e^{\delta(\chi)/|t|}\right\}^\epsilon$ for some fixed $\epsilon>0$, then
\be\label{lchil1 e2}
\left|L_y^{(k)}(s,\chi)\right|\ll k!(c_\epsilon\log y)^k.
\ee
\end{lemma}

\begin{proof} Set $z=\max\{y,q^4,V_t^{100}\}$ and note that
\bes\label{lchi e1}
(-1)^kL_y^{(k)}(s,\chi) =  \sum_{\substack{n>z\\P^-(n)>y}} \frac{\chi(n)(\log n)^k}{n^s} + O\left(\frac{(\log z)^{k+1}}{\log y}\right).
\ees
Moreover, Lemma~\ref{chinit} implies that
$$
\sum_{\substack{n\le u\\P^-(n)>y}}\chi(n)n^{-it} = \frac{\delta(\chi) \phi(q) }{q} \frac{u^{1-it}} {1-it} \prod_{\substack{p\le y\\p\nmid q}} \left(1-\frac1p\right) +R_t(u)
$$
with $R_t(u)\ll u^{1-1/(30\log z)}/\log y$ for $u\ge z$. Consequently, we have that
\bes\bsp
\sum_{\substack{n>z\\P^-(n)>y}} \frac{\chi(n)(\log n)^k}{n^s}
&= \int_z^\infty \frac{(\log u)^k}{u^\sigma} d\left( \frac{\delta(\chi) \phi(q) }{q} \frac{u^{1-it}} {1-it} \prod_{\substack{p\le y\\p\nmid q}} \left(1-\frac1p\right) + R_t(u)\right) \\
&= \frac{\delta(\chi)\phi(q)}q \prod_{\substack{p\le y\\p\nmid q}}\left(1-\frac1p\right) \int_z^\infty\frac{(\log u)^k}{u^s}du
+\int_z^\infty\frac{(\log u)^k}{u^\sigma}d R_t(u).
\end{split}\ees
Since
\bes\bsp
\int_z^\infty\frac{(\log u)^k}{u^\sigma}d R_t(u) &= -\frac{(\log z)^kR_t(z)}{z^\sigma} + \int_z^\infty \frac{(\log u)^{k-1}(\sigma\log u-k)}{u^{\sigma+1}}R_t(u)du \\
&\ll\frac{(\log z)^k}{\log y}+\frac{\sigma+k}{\log y}\int_z^\infty\frac{(\log u)^k}{u^{\sigma+1/(30\log z)}}du\\
&\le\frac{(\log z)^k}{\log y}+\frac{\sigma+k}{z^{\sigma-1}\log y}\int_z^\infty\frac{(\log u)^k}{u^{1+1/(30\log z)}}du
	\ll \frac{ (k+1)!(30\log z)^{k+1} }{\log y},
\end{split}\ees
by partial summation, and
$$
\int_z^\infty\frac{(\log u)^k}{u^s}du=\int_1^\infty\frac{(\log u)^k}{u^s}du+O((\log z)^{k+1})=\frac{k!}{(s-1)^{k+1}}+O((\log z)^{k+1}),
$$
by observing that $(\log u)^mu^{-s}=(\log u)^m\frac{d(u^{1-s}/(1-s))}{du}$ for all $m\ge0$ and integrating by parts $k$ times, relation~\eqref{lchil1 e1} follows. Finally, relation~\eqref{lchil1 e2} is a direct consequence of relation~\eqref{lchil1 e1}, since $|s-1|\ge|t|\ge\epsilon\cdot\delta(\chi)/\log y$ under the assumption that $y\ge e^{\epsilon\cdot\delta(\chi)/|t|}$. This completes the proof of the lemma.
\end{proof}

The next lemma provides a lower bound on $L_y(s,\chi)$ close to the line $\Re(s)=1$.

\begin{lemma}\label{lchil2} Fix $\epsilon\in(0,1]$. Let $\chi$ be a Dirichlet character modulo $q$, $s=\sigma+it$ with $\sigma>1$ and $t\in\SR$, and $y\ge qV_t$. If $|t|\ge\epsilon/\log y$ or if $\chi$ is complex, then we have that $|L_y(s,\chi)|\asymp_\epsilon1$. Finally, if $\chi$ is a non-principal, real character, and $|t|\le1/\log y$, then $|L_y(s,\chi)|\gg L_y(1,\chi)$.
\end{lemma}

\begin{proof} First, assume that either $|t|\ge\epsilon/\log y$ or $\chi$ is complex. Equivalently, $|t|\ge\epsilon \delta(\chi^2)/\log y$. Note that for every $x\ge y$ we have that
\be\label{lbound e1}\bsp
2\cdot\SD(\chi(n),\mu(n)n^{it};y,x)&=\SD(1,\overline{\chi}(n)\mu(n)n^{it};y,x) + \SD(1,\overline{\chi}(n)\mu(n)n^{it};y,x)\\
&\ge\SD(1,\overline{\chi}^2(n)\mu^2(n)n^{2it};y,x)=\SD(\chi^2(n),n^{2it};y,x),
\end{split}\ee
by Lemma~\ref{triangle}. Mertens's estimate on $\sum_{p\le t}1/p$ and Lemma~\ref{l1} imply that
\bes\bsp
\SD^2(\chi^2(n),n^{2it};y,x)&=\log\left(\frac{\log x}{\log y}\right)+O(1)-\sum_{y<p\le x}\frac{\Re(\chi^2(p)p^{-2it})}p\\
&= \log\left(\frac{\log x}{\log y}\right)+O(1)-\log\left|L_y\left(1+\frac1{\log x}+2it,\chi^2\right)\right|.
\end{split}\ees
So, applying relation~\eqref{lchil1 e2} with $k=0$, we find that
$$
\SD^2(\chi^2(n),n^{2it};y,x)\ge\log\left(\frac{\log x}{\log y}\right)+O(1).
$$
Inserting this estimate into~\eqref{lbound e1}, we deduce that
$$
\SD^2(\chi(n),\mu(n)n^{it};y,x)\ge\frac14\log\left(\frac{\log x}{\log y}\right)+O(1)\quad\left(x\ge y\ge\max\left\{qV_t,e^{\epsilon\cdot\delta(\chi^2)/|t|}\right\}\right).
$$
The above inequality, and relation~\eqref{lchil1 e2} with $k=1$, allow us to apply Lemma~\ref{dist1} with $y_0=y_1=y$ and any $y_2>y$. So we conclude that
\be\label{lchie1}
\left|\sum_{y<p\le x}\frac{\chi(p)}{p^{1+it}}\right|\ll1 \quad\left(x\ge y\ge\max\left\{qV_t,e^{\epsilon\cdot\delta(\chi^2)/|t|}\right\}\right).
\ee
The above relation for $x=\max\{e^{1/(\sigma-1)},y\}$ and Lemma~\ref{l1} imply that $|L_y(s,\chi)|\asymp1$, which completes the proof of the first part of the lemma.

Finally, assume that $\chi$ is a real, non-principal character, and that $|t|\le1/\log y$. Let $x=\max\{e^{1/(\sigma-1)},y\}$ and $z=\min\{x,e^{1/|t|}\}\ge y\ge qV_t$. Then we have that
$$
\left|\sum_{z<p\le x}\frac{\chi(p)}{p^{1+it}}\right|\ll1,
$$
trivially if $z=x$, and by relation~\eqref{lchie1} with $z$ in place of $y$ otherwise, which holds, since in this case $z=e^{1/|t|}\ge e^{\delta(\chi^2)/|t|}$.
So we deduce that
$$
\sum_{y<p\le x}\frac{\chi(p)}{p^{1+it}}
= \sum_{y<p\le z}\frac{\chi(p)}{p^{1+it}} +O(1)
= \sum_{y<p\le z}\frac{\chi(p)+O(|t|\log p)}p + O(1)
= \sum_{y<p\le z}\frac{\chi(p)}p + O(1).
$$
Finally, for every $w\ge z\ge qV_t$, we have that
$$
\sum_{z<p\le w}\frac{\chi(p)}p=\log\left|L_z\left(1+\frac1{\log w},\chi\right)\right|+O(1)\le O(1),
$$
by Lemma~\ref{l1}, and relation~\eqref{lchil1 e2} with $k=0$. So
$$
\sum_{y<p\le x}\frac{\Re(\chi(p)p^{-it})}p = \sum_{y<p\le z}\frac{\chi(p)}p + O(1) \ge \sum_{y<p\le w} \frac{\chi(p)}p+O(1)\quad(w\ge z).
$$
Lemma~\ref{l1} then implies that $|L_y(s,\chi)|\gg L_y(1+1/\log w,\chi)$ for all $w\ge z$. Letting $w\to\infty$ completes the proof of the last part of the lemma too.
\end{proof}

Finally, we prove an estimate for the derivatives of $\frac{L_y'}{L_y}(s,\chi)$, which will be key in the proof of Theorem~\ref{pntap}.

\begin{lemma}\label{lchil3} Let $\chi$ be a Dirichlet character modulo $q$ and $s=\sigma+it$ with $\sigma>1$ and $t\in\SR$. For every $k\in\SN$ we have that
$$
\left|\left(\frac{L'}{L}\right)^{(k-1)}(s,\chi)+\frac{\delta(\chi)(-1)^{k-1}(k-1)!}{(s-1)^k}\right|
\ll\left(\frac{ck\log(qV_t)}{\delta(\chi)+(1-\delta(\chi))|L_{qV_t}(s,\chi)|}\right)^k.
$$
\end{lemma}

\begin{proof} Set $y=qV_t$ and fix some constant $\epsilon$ to be chosen later. We separate three cases.

\medskip

\textit{Case 1: $\sigma\ge1+\epsilon/\log y$}. Note that $\delta(\chi)+(1-\delta(\chi))|L_{qV_t}(s,\chi)|\ll1$, trivially if $\delta(\chi)=1$, and by relation~\eqref{lchil1 e2} with $k=0$ if $\delta(\chi)=0$. Since we also have that
$$
\left|\left(\frac{L'}{L}\right)^{(k-1)}(s,\chi)+\frac{\delta(\chi)(-1)^{k-1}(k-1)!}{(s-1)^k}\right|
\le\sum_{n=1}^\infty\frac{\Lambda(n)(\log n)^{k-1}}{n^{1+\epsilon/\log y} }+\frac{(k-1)!}{(\epsilon/\log y)^k} \ll (c_1k\log y)^k
$$
for some $c_1=c_1(\epsilon)$, the lemma follows.

\medskip

\textit{Case 2: $|t|>\epsilon/\log y$}. Note that
\be\label{lchie3}
\left(\frac{L'}{L}\right)^{(k-1)}(s,\chi) = O\left((c_2k\log y)^k\right) + \left(\frac{L_y'}{L_y}\right)^{(k-1)}(s,\chi).
\ee
Furthermore, relation~\eqref{lchil1 e2} implies that $|L_y^{(j)}(s,\chi)|\le j!(c_3\log y)^j$ for all $j\in\SN$, for some $c_3=c_3(\epsilon)$. Additionally, we have that $|L_y(s,\chi)|\asymp_\epsilon1$ by Lemma~\ref{lchil2}. So Lemma~\ref{derlemma} applied to $F(s)=L_y(s,\chi)$ yields that the right hand side
of~\eqref{lchie3} is $\ll(c_4k\log y)^k$ for some $c_4=c_4(\epsilon)$, and the lemma follows (note that in this case $|s-1|\ge |t|\ge \epsilon/ \log y$).

\medskip

\textit{Case 3: $|s-1|\le2\epsilon/\log y$}. Let
$$
F(s)=(s-1)^{\delta(\chi)}L_y(s,\chi)\prod_{p\le y}\left(1-\frac1p\right)^{-\delta(\chi)},
$$
and observe that
$$
F^{(j)}(s) = \left( (s-1)^{\delta(\chi)} L_y^{(j)}(s,\chi) + \delta(\chi) j L_y^{(j-1)}(s,\chi) \right) \prod_{p\le y} \left( 1- \frac1p\right)^{-\delta(\chi)} \ll j!(c_5\log y)^j,
$$
for all $j\in\SN$, by relation~\eqref{lchil1 e1}. So Lemma~\ref{derlemma} implies that
$$
\left(\frac{L_y'}{L_y}\right)^{(k-1)}(s,\chi) + \frac{\delta(\chi)(-1)^{k-1}(k-1)!}{(s-1)^k} = \left(\frac{F'}{F}\right)^{(k-1)}(s) \ll \left(\frac{c_6k\log y}{\min\{|F(s)|,1\}}\right)^k.
$$
Together with~\eqref{lchie3}, the above estimate reduces the desired result to showing that
\be\label{lchil3 e1}\min\{|F(s)|,1\}\asymp\delta(\chi)+(1-\delta(\chi))|L_y(s,\chi)|.
\ee
If $\delta(\chi)=0$, then~\eqref{lchil3 e1} holds, since $|F(s)|=|L_{y}(s,\chi)|\ll1$, by relation~\eqref{lchil1 e2} with $k=0$. Lastly, if $\delta(\chi)=1$, then $|F(s)|=1+O(|s-1|\log y)=1+O(\epsilon)$ by relation~\eqref{lchil1 e1} with $k=0$ (note that $\prod_{p\le y,\,p|q}(1-1/p)=\phi(q)/q$, since $y>q$ by assumption). So choosing $\epsilon$ small enough (independently of $k$, $q$ and $s$), we find that $|F(s)|\asymp1$, that is,~\eqref{lchil3 e1} is satisfied in this case too. This completes the proof of~\eqref{lchil3 e1} and hence of the lemma.
\end{proof}

%%%%%%%%%%%%%%%%%%%%%%%%%%%%%%%%%%%%%%%%%%%%%%%%%%%%%%%%%%%%%%%%%%%%%%%%%%%%%%%%%%%%%%%%%%%%%%%%%%%%%%%%%%%%%%%%%%%%%%%%%%%%
%%%%%%%%%%%%%%%%%%%%%%%%%%%%%%%%%%%%%%%%%%%%%%%%%%%%%%%%%%%%%%%%%%%%%%%%%%%%%%%%%%%%%%%%%%%%%%%%%%%%%%%%%%%%%%%%%%%%%%%%%%%%
%%%%%%%%%%%%%%%%%%%%%%%%%%%%%%%%%%%%%%%%%%%%%%%%%%%%%%%%%%%%%%%%%%%%%%%%%%%%%%%%%%%%%%%%%%%%%%%%%%%%%%%%%%%%%%%%%%%%%%%%%%%%
%%%%%%%%%%%%%%%%%%%%%%%%%%%%%%%%%%%%%%%%%%%%%%%%%%%%%%%%%%%%%%%%%%%%%%%%%%%%%%%%%%%%%%%%%%%%%%%%%%%%%%%%%%%%%%%%%%%%%%%%%%%%
%%%%%%%%%%%%%%%%%%%%%%%%%%%%%%%%%%%%%%%%%%%%%%%%%%%%%%%%%%%%%%%%%%%%%%%%%%%%%%%%%%%%%%%%%%%%%%%%%%%%%%%%%%%%%%%%%%%%%%%%%%%%

\section{Siegel's theorem}\label{siegelproof}

In this section we prove Theorem~\ref{siegel}. In fact, we will show a more precise result, originally due to Tatuzawa~\cite{Ta}, from which Theorem~\ref{siegel} follows immediately:

\begin{theorem}\label{siegel2} Let $\epsilon\in(0,1]$. With at most one exception, for all real, non-principal, primitive Dirichlet characters $\chi$, we have that $L(1,\chi)\gg\epsilon q^{-\epsilon}$, where $q$ denotes the conductor of $\chi$; the implied constant is effectively computable.
\end{theorem}

\begin{proof}[Deduction of Theorem~\ref{siegel} from Theorem~\ref{siegel2}] Fix some $\epsilon\in(0,1)$ and let $\psi_\epsilon$ be the possible exceptional character from Theorem~\ref{siegel2}.
Consider a real, non-principal Dirichlet character $\chi$ modulo $q$, and let $\chi_1$ modulo $q_1$ be the primitive character inducing $\chi$. We have that
$$
L(1,\chi)=L(1,\chi_1)\prod_{p|q}\left(1-\frac{\chi_1(p)}p\right)\ge L(1,\chi_1) \cdot \frac{\phi(q)}q \gg\frac{L(1,\chi_1)}{\log q}.
$$
So
$$
L(1,\chi)\gg\frac1{\log q}\cdot
\begin{cases}L(1,\psi_\epsilon) &\text{if}\ \chi_1=\psi_\epsilon,\cr
\epsilon q^{-\epsilon} &\text{if}\ \chi_1\neq\psi_\epsilon.
\end{cases}
$$
Moreover, Dirichlet's class number formula~\cite[p. 49-50]{Dav} implies that $L(1,\psi_\epsilon)>0$ (see also~\cite[p. 37]{ik} for a different proof of this statement, in the spirit of the proof of Lemma~\ref{zerol1} below). So, in any case, we find that $L(1,\chi)\gg_\epsilon q^{-\epsilon}/\log q$. Since this inequality holds for every $\epsilon\in(0,1)$, Theorem~\ref{siegel} follows.
\end{proof}

Before we give the proof of Theorem~\ref{siegel2}, we state a preliminary lemma. The idea behind its proof can be traced back to~\cite{Pi73,Pi76}. Note that some of its assumptions concern the behavior of $L(s,f)$ inside the critical strip. However, the assumption that $\sum_{n\le x}f(n)\ll x^{4/5}\log x$, together with partial summation, guarantees that $L(s,f)$ converges (conditionally) for all $s$ with $\Re(s)>4/5$. Hence we still use only elementary facts about Dirichlet series and no analytic continuation.
\begin{lemma}\label{zerol1} Let $c\ge1$, $r\in\SN$ and $Q\ge2$. Consider a multiplicative function $f:\SN\to\SR$ such that $0\le(1*f)(n)\le\tau_r(n)$ for all integers $n$, and
$$
\left|\sum_{n\le x}f(n)\right|\le cx^{4/5}\log x\quad(x\ge Q).
$$
\begin{enumerate}
\item  If $L(1-\eta,f)\ge0$ for some $\eta\in[0,1/20]$, then $L(1,f)\gg_{c,r}\eta/Q^{2\eta}$.
\item There is a constant $c_0$ depending at most on $c$ and $r$ such that $L(\sigma,f)$ has at most one zero in the interval $[1-c_0/\log Q,1]$.
\end{enumerate}
\end{lemma}

\begin{proof} Before we begin, we record some estimates, which follow by partial summation. For $A_2\ge A_1\ge Q$, we have that
\be\label{sum f(a)}
\sum_{A_1<a\le A_2}\frac{f(a)}{a^{1-\eta}}\ll_c\frac{\log A_1}{A_1^{1/5-\eta}} \quad\text{and}\quad \sum_{A_1<a\le A_2}\frac{f(a)\log a}{a^{1-\eta}}\ll_c\frac{(\log A_1)^2}{A_1^{1/5-\eta}},
\ee
and, for $B\ge1$, we have that
\be\label{1/b}
\sum_{b\le B}\frac1{b^{1-\eta}}=\frac{B^\eta-1}{\eta}+\gamma_\eta+O(B^{\eta-1}),\quad\text{where}\ \gamma_\eta=1-(1-\eta)\int_1^\infty\frac{\{u\}}{u^{2-\eta}}du,
\ee
and
\be\label{logb/b}
\sum_{b\le B}\frac{\log b}{b^{1-\eta}}=\frac{B^\eta\log B}{\eta}-\frac{B^\eta-1}{\eta^2}+\gamma'_\eta+O\left(\frac{\log B}{B^{1-\eta}}\right),\quad\text{where}\ \gamma'_\eta=\int_1^\infty\frac{\{u\}(1-(1-\eta)\log u)}{u^{2-\eta}}du.
\ee

\medskip

(a) For $x\ge Q^2$ we have that
\bes\bsp
S_1&:=\sum_{n\le x}\frac{(1*f)(n)}{n^{1-\eta}}=\sum_{a\le\sqrt{x}}\frac{f(a)}{a^{1-\eta}}\sum_{b\le x/a}\frac1{b^{1-\eta}}+\sum_{b\le\sqrt{x}}\frac1{b^{1-\eta}}\sum_{\sqrt{x}<a\le x/b}\frac{f(a)}{a^{1-\eta}}\\
&=\sum_{a\le\sqrt{x}}\frac{f(a)}{a^{1-\eta}}\left(\frac{(x/a)^\eta-1}\eta+\gamma_\eta+O\left((x/a)^{\eta-1}\right)\right)
+O_c\left(\frac{\log x}{x^{1/10-\eta/2}}\sum_{b\le\sqrt{x}}\frac1{b^{1-\eta}}\right)\\
&=\sum_{a\le\sqrt{x}}\frac{f(a)}{a^{1-\eta}}\frac{(x/a)^\eta-1}\eta+\gamma_\eta L(1-\eta,f)+O_{c,r}\left(\frac{\log^2x}{x^{1/10-\eta}}\right),
\end{split}\ees
by relations~\eqref{sum f(a)} and~\eqref{1/b}, since $|f|=|\mu*(1*f)|\le\tau_{r+1}$. Finally, for $A>\sqrt{x}$ we have that
$$
\sum_{\sqrt{x}<a\le A}\frac{f(a)}{a^{1-\eta}}\frac{(x/a)^\eta-1}\eta
= -\int_{\sqrt{x}}^A\frac{(x/u)^\eta-1}\eta d\left(\sum_{u<a\le A}\frac{f(a)}{a^{1-\eta}}\right)\ll_c\frac{\log^2x}{x^{1/10-\eta}},
$$
by relation \eqref{sum f(a)} and integration by parts. Consequently,
$$
S_1 = \frac{x^\eta}\eta L(1,f) + \left(\gamma_\eta-\frac1\eta\right)L(1-\eta,f)
+O_{c,r}\left(\frac{\log^2x}{x^{1/10-\eta}}\right).
$$
Note that $S_1\ge 1$ because $(1*f)(1)=1$ and $(1*f)(n) \ge0$ for all $n>1$ by assumption. Since $L(1-\eta,f)\ge0$ and $\gamma_\eta<1<1/\eta$ for $\eta\in(0,1)$, we readily find that $L(1,f)\gg\eta/Q^{2\eta}$ by taking $x=c_1Q^2$ for some sufficiently large constant $c_1$ that depends at most on $c$ and $r$.

\medskip

(b) Let $\eta\in[0,1/20]$ and $x\ge Q^2$. We use a similar argument as in the proof of part (a) to compute
\[
S_2:= \sum_{n\le x}\frac{(1*f)(n)\log n}{n^{1-\eta}}=\sum_{a\le\sqrt{x}}\frac{f(a)}{a^{1-\eta}}\sum_{b\le x/a}\frac{\log(ab)}{b^{1-\eta}}+\sum_{b\le\sqrt{x}}\frac1{b^{1-\eta}}\sum_{\sqrt{x}<a\le x/b}\frac{f(a)\log(ab)}{a^{1-\eta}}.
\]
Relations~\eqref{1/b} and~\eqref{logb/b} imply that, for $1\le a\le x$, we have that
\[
\sum_{b\le x/a}\frac{\log(ab)}{b^{1-\eta}} = (\log a)\sum_{b\le x/a}\frac{1}{b^{1-\eta}} + \sum_{b\le x/a}\frac{\log b}{b^{1-\eta}}= g(a) + O\left(\frac{\log(2x)}{(x/a)^{1-\eta}}\right),
\]
where
\begin{align}
g(a)
&= (\log a)\left(\frac{(x/a)^\eta-1}{\eta}+\gamma_\eta\right)+\frac{\eta(x/a)^\eta\log(x/a) - (x/a)^\eta+1}{\eta^2}+\gamma_\eta' \label{g_x(a)1}\\
&= (\log x)\frac{(x/a)^\eta}{\eta}+(\log a)\left(\gamma_\eta-\frac1\eta\right) - \frac{(x/a)^\eta-1}{\eta^2}+\gamma_\eta' \label{g_x(a)2}.
\end{align}
By~\eqref{g_x(a)1}, we find that
\be\label{g_x(a)}
g(a)\ll \log^2(2a) + (x/a)^\eta\log^2x,
\ee
uniformly in $a\ge1$, $x\ge1$ and $\eta\in[0,1/20]$. Moreover, differentiating~\eqref{g_x(a)2} with respect to $a$, we deduce that
\be\label{g_x'(a)}\begin{split}
g'(a)
&= -\frac{(\log x)(x/a)^\eta}{a}+\frac1a\left(\gamma_\eta-\frac1\eta\right)+\frac{(x/a)^\eta}{\eta a} \\
&= -\frac{(\log x)(x/a)^\eta}{a}+\frac{\gamma_\eta}a+\frac{(x/a)^\eta-1}{\eta a} \ll\frac{\log(2a)+(x/a)^\eta\log x}{a},
\end{split}\ee
uniformly in $a\ge1$, $x\ge1$ and $\eta\in[0,1/20]$. Therefore
\bes\bsp
S_2&=\sum_{a\le\sqrt{x}}\frac{f(a)}{a^{1-\eta}}\left(g(a)+O\left(\frac{\log(x/a)}{(x/a)^{1-\eta}}\right)\right)
+O_c\left(\frac{(\log x)^3}{x^{1/10-\eta}}\right)\\
&=\sum_{a\le\sqrt{x}}\frac{f(a)g(a)}{a^{1-\eta}}+O_{c,r}\left(\frac{(\log x)^3}{x^{1/10-\eta}}\right)
= \sum_{a=1}^\infty\frac{f(a)g(a)}{a^{1-\eta}}+O_{c,r}\left(\frac{(\log x)^3}{x^{1/10-\eta}}\right)
\end{split}\ees
for all $x\ge Q^2$, where the first equality is a consequence of relation \eqref{sum f(a)} and the third equality follows by relations \eqref{g_x(a)} and \eqref{g_x'(a)} together with partial summation. Furthermore, we have that
\[
\sum_{a=1}^\infty\frac{f(a)g(a)}{a^{1-\eta}} = \frac{x^\eta(\eta\log x-1)}{\eta^2}L(1,f)+\left(\frac1\eta-\gamma_\eta\right)L'(1-\eta,f)+\left(\frac1{\eta^2}+\gamma'_\eta\right)L(1-\eta,f).
\]
Since $S_2\ge0$, by our assumption that $1*f\ge0$, we conclude that
\be\label{L'(1,f)}
0\le \frac{x^\eta(\eta\log x-1)}{\eta^2}L(1,f)+\left(\frac1\eta-\gamma_\eta\right)L'(1-\eta,f)+\left(\frac1{\eta^2}+\gamma'_\eta\right)L(1-\eta,f) +O_{c,r}\left(\frac{(\log x)^3}{x^{1/10-\eta}}\right).
\ee

Now, assume that $L(1-\eta,f)=0$ for some $\eta\in[0,c_0/\log Q]$, where $c_0$ a small constant to be chosen later, then setting $x=e^{1/(2\eta)}$ in~\eqref{L'(1,f)}, we find that
\[
0\le -\frac{e^{1/2}}{2\eta^2}L(1,f)+\left(\frac1\eta-\gamma_\eta\right)L'(1-\eta,f)+O_{c,r}\left(\frac{(1/\eta)^3}{e^{1/(20\eta)}}\right).
\]
Also, part (a) implies that $L(1,f)\gg\eta/Q^{2\eta}\gg\eta$. So there is some constant $c_2=c_2(c,r)$ such that
\[
\left(\frac1\eta-\gamma_\eta\right)L'(1-\eta,f)\ge \frac{c_2}{\eta} + O_{c,r}\left(\frac{(1/\eta)^3}{e^{1/(20\eta)}}\right) \ge \frac{c_2}{2\eta},
\]
provided that $c_0$ is large enough. Since $\gamma_\eta<1<1/\eta$ for $\eta\in(0,1)$, we conclude that $L'(1-\eta,f)>0$.

We are now ready to complete the proof of part (b): let
\[
\mathscr{Z}=\{\beta\in[1-c_0/\log Q,1]:L(\beta,f)=0\}
\]
and note that $\mathscr{Z}$ is a closed set under the Euclidean topology. Since $L'(\beta,f)>0$ whenever $\beta\in \mathscr{Z}$ from the discussion in the above paragraph, the set $\mathscr{Z}$ has no accumulation points ($L'(\sigma,f)$ is continuous for $\sigma>4/5$, so around every element of $\mathscr{Z}$ there is an open neighborhood where $L'(\sigma,f)>0$ and, as a result, where $L(\sigma,f)$ is strictly increasing\footnote{This is a consequence of the Mean Value Theorem, but it can be also shown directly: Integrating term by term, we have that $\sum_{n=1}^Nf(n)n^{-\sigma_2}-\sum_{n=1}^Nf(n)n^{-\sigma_1} = -\int_{\sigma_1}^{\sigma_2}\sum_{n=1}^N f(n)(\log n)n^{-u}du$. So letting $N\to\infty$ yields the formula $L(\sigma_2,f)-L(\sigma_1,f) = \int_{\sigma_1}^{\sigma_2}L'(\sigma,f) d\sigma$ for $\sigma_2>\sigma_1>4/5$, since the series $\sum_{n=1}^Nf(n)(\log n)/n^\sigma$ converges uniformly in $[\sigma_1,\sigma_2]$, as it can be seen by partial summation. Hence if $L'(\sigma,f)>0$ for all $\sigma\in[\sigma_1,\sigma_2]$, then $L(\sigma_2,f)>L(\sigma_1,f)$.}). Now, assume that $\beta_1<\beta_2$ are two consecutive elements of $\mathscr{Z}$, that is to say, there is no other element of $\mathscr{Z}$ in $(\beta_1,\beta_2)$. Then $L(\sigma,f)$ is strictly increasing in a neighborhood of $\beta_1$ and of $\beta_2$. In particular, $L(\sigma,f)$ is positive in a neighborhood to the right of $\beta_1$ and negative in a neighborhood to the left of $\beta_2$. But then the Intermediate Value Theorem implies that there must another zero of $L(\sigma,f)$ in the interval $(\beta_1,\beta_2)$, which contradicts the choice of $\beta_1$ and $\beta_2$. This shows that $\mathscr{Z}$ can contain at most one element, thus completing the proof of the lemma.
\end{proof}

\begin{proof}[Proof of Theorem~\ref{siegel2}] Fix $\epsilon\in(0,1]$ and let $\mathscr{C}$ be the set of all real, non-principal, primitive Dirichlet characters. Note that if $\chi\pmod q$ is an element of $\mathscr{C}$ for which $L(\sigma,\chi)$ has no zeroes in $[1-\epsilon/40,1]$, then $L(1-\epsilon/40,\chi)>0$, by continuity. So Lemma~\ref{zerol1}(a) with $Q=q^{5/4}$ implies that $L(1,\chi)\ge c_1\epsilon /q^{\epsilon}$, for some absolute constant $c_1>0$.

Now, assume that there are two distinct elements of $\mathscr{C}$, say $\chi_1\pmod{q_1}$ and $\chi_2\pmod{q_2}$, such that $L(1,\chi_j)<c_2\epsilon/q_j^\epsilon$ for $j\in\{1,2\}$, where $c_2$ is some constant in $(0,c_1]$ to be chosen later. Then the discussion in the above paragraph implies that $L(\sigma,\chi_1)$ and $L(\sigma,\chi_2)$ both vanish somewhere in $[1-\epsilon/40,1]$, say at $1-\eta_1$ and $1-\eta_2$, respectively. Set $q=\max\{q_1,q_2\}$ and $f=\chi_1*\chi_2*\chi_1\chi_2$. We shall apply Lemma~\ref{zerol1} to $f$ but, first, we need to bound its partial sums. We do so by Dirichlet's hyperbola method: Note that
\bes\bsp
\sum_{n\le x}(\chi_1*\chi_2)(n) &= \sum_{a\le\sqrt{x}}\chi_1(a)\sum_{b\le x/a}\chi_2(b)+\sum_{b\le\sqrt{x}}\chi_2(b)\sum_{\sqrt{x}<a\le x/b}\chi_1(a)
\ll \sum_{a\le\sqrt{x}}q_2+\sum_{b\le\sqrt{x}}q_1 \le2q\sqrt{x}.
\end{split}\ees
So for $x\ge q^{10}$ we have that
\bes\bsp
\sum_{n\le x}f(n) & = \sum_{k\le (x/q)^{2/3} } (\chi_1*\chi_2)(k) \sum_{m\le x/k}\chi_1\chi_2(m) + \sum_{m\le x^{1/3}q^{2/3}}\chi_1\chi_2(m)\sum_{(x/q)^{2/3}<k\le x/m}(\chi_1*\chi_2)(k) \\
&\ll \sum_{k\le (x/q)^{2/3}}\tau_2(k)q^2 + \sum_{m \le x^{1/3}q^{2/3}} q\sqrt{x/m} \ll q^{4/3}x^{2/3}\log x\le x^{4/5}\log x,
\end{split}\ees
that is to say, the hypotheses of Lemma~\ref{zerol1} are satisfied with $Q=q^{10}$ and $r=4$. Since $L(\sigma,f)=L(\sigma,\chi_1)L(\sigma,\chi_2)L(\sigma,\chi_1\chi_2)$ for $\sigma>2/3$, a formula which can be proven using Dirichlet's hyperbola method in a similar fashion as above, we deduce that $L(1-\eta_1,f)=L(1-\eta_2,f)=0$. So Lemma~\ref{zerol1}(b) implies that $\eta:=\max\{\eta_1,\eta_2\}\gg1/\log q$. Moreover, Lemma~\ref{zerol1}(a) and the inequality $e^x\ge x^2/2$, $x>0$, imply that
\be\label{siegel e1}
L(1,f)\gg\frac{\eta}{q^{20\eta}}\ge\frac{\eta}{q^{\epsilon/2}}\gg\frac1{q^{\epsilon/2}\log q}\ge\frac{\epsilon^2\log q}{4q^{\epsilon}}.
\ee
However, our assumption on $\chi_1$ and $\chi_2$ and the standard bound $L(1,\chi_1\chi_2)\ll\log q$,
which follows by partial summation, imply that
$$
L(1,f)=L(1,\chi_1)L(1,\chi_2)L(1,\chi_1\chi_2)\ll \frac{c_2\epsilon}{q_1^\epsilon}\frac{c_2\epsilon}{q_2^\epsilon}(\log q)\le\frac{c_2^2\epsilon^2\log q}{q^\epsilon},
$$
which contradicts~\eqref{siegel e1} if $c_2$ is chosen to be small enough. So we conclude that there cannot be $\chi_1\pmod{q_1}$ and $\chi_2\pmod{q_2}$, distinct elements of $\mathscr{C}$, such that $L(1,\chi_j)<c_2\epsilon/q_j^\epsilon$ for both $j=1$ and $j=2$. This completes the proof of Theorem~\ref{siegel2}.
\end{proof}

%%%%%%%%%%%%%%%%%%%%%%%%%%%%%%%%%%%%%%%%%%%%%%%%%%%%%%%%%%%%%%%%%%%%%%%%%%%%%%%%%%%%%%%%%%%%%%%%%%%%%%%%%%%%%%%%%%%%%%%%%%%%
%%%%%%%%%%%%%%%%%%%%%%%%%%%%%%%%%%%%%%%%%%%%%%%%%%%%%%%%%%%%%%%%%%%%%%%%%%%%%%%%%%%%%%%%%%%%%%%%%%%%%%%%%%%%%%%%%%%%%%%%%%%%
%%%%%%%%%%%%%%%%%%%%%%%%%%%%%%%%%%%%%%%%%%%%%%%%%%%%%%%%%%%%%%%%%%%%%%%%%%%%%%%%%%%%%%%%%%%%%%%%%%%%%%%%%%%%%%%%%%%%%%%%%%%%
%%%%%%%%%%%%%%%%%%%%%%%%%%%%%%%%%%%%%%%%%%%%%%%%%%%%%%%%%%%%%%%%%%%%%%%%%%%%%%%%%%%%%%%%%%%%%%%%%%%%%%%%%%%%%%%%%%%%%%%%%%%%
%%%%%%%%%%%%%%%%%%%%%%%%%%%%%%%%%%%%%%%%%%%%%%%%%%%%%%%%%%%%%%%%%%%%%%%%%%%%%%%%%%%%%%%%%%%%%%%%%%%%%%%%%%%%%%%%%%%%%%%%%%%%

\section{Proof of Theorem~\ref{pntap}}\label{proofs}

As we mentioned in the introduction, we shall deduce Theorem~\ref{pntap} from the following more general result.

\begin{theorem}\label{pntap2} For $x\ge1$ and $(a,q)=1$ we have that
\begin{align*}
\sum_{\substack{n\le x\\n\equiv a\pmod q}}\Lambda(n)
=\frac x{\phi(q)}& + O\left(xe^{-c_1(\log x)^{3/5}(\log\log x)^{-1/5}}+ \sqrt{\log x}\left( x^{1-\frac{c_2}{\log q} } + \frac{\tau_2(q) x^{1-c_3\eta(q)} } {\phi(q)} \right) \right),
\end{align*}
where $\eta(q)=\min\left\{ L_q(1,\chi) / \log(3q) :\chi\ \text{real and non-principal character mod}\ q\right\}$.
\end{theorem}

\begin{proof}[Deduction of Theorem~\ref{pntap} from Theorem~\ref{pntap2}]
Fix $A\ge1$. Note that if $\chi$ is real and non-principal, then
\begin{align*}
\frac{L_q(1,\chi)}{\log q}  =\lim_{\sigma\to1^+} \left\{\frac{L(\sigma,\chi)}{\log q} \prod_{p\le q} \left(1-\frac{\chi(p)}{p^\sigma}\right)\right\}
 =\frac{L(1,\chi)}{\log q} \prod_{p\le q} \left( 1-\frac{\chi(p)}p \right)
\gg \frac{L(1,\chi)} {\log^2q} &\gg_A  q^{-\frac1{3A}}.
\end{align*}
by Mertens's estimate and Theorem~\ref{siegel}. So if $x\ge e^{q^{1/A}}$, then we find that $\eta(q)\gg_A q^{-1/(3A)}\ge(\log x)^{-1/3}$. Combining this estimate with Theorem~\ref{pntap2} completes the proof of Theorem~\ref{pntap}.
\end{proof}

\begin{proof}[Proof of Theorem~\ref{pntap2}] For every $x\ge1$ and $(a,q)=1$, the orthogonality of characters implies that
\be\label{proofe4}
\sum_{\substack{n\le x\\n\equiv\,a\,(\text{mod}\ q)}}\Lambda(n)-\frac x{\phi(q)} =
\frac1{\phi(q)}\sum_{\chi\ (\text{mod}\ q)}\overline{\chi}(a)\sum_{n\le x}(\chi(n)\Lambda(n)-\delta(\chi)) + O\left(\frac 1{\phi(q)}\right).
\ee
Fix for the moment a character $\chi$ modulo $q$ and set $\Lambda_\chi(n)=\chi(n)\Lambda(n)-\delta(\chi)$ and
$$
F_\chi(s)=\sum_{n=1}^\infty\frac{\Lambda_\chi(n)}{n^s}=-\frac{L'}{L}(s,\chi)-\delta(\chi)\zeta(s).
$$
We claim that, for $k\in\SN$ and $s=\sigma+it$ with $\sigma>1$ and $t\in\SR$, we have that
\be\label{pf e1}
F_\chi^{(k-1)}(s) \ll \begin{cases} (c_4k\log(qV_t))^k &\text{if}\ |t|\ge1/\log(3q),\cr
\displaystyle (c_4k\log(3q))^k M(\chi)^{-k}&\text{otherwise},
\end{cases}
\ee
where
$$
M(\chi)=
\begin{cases}
L_q(1,\chi)&\text{if}\ \chi\ \text{is real, non-principal},\cr
1&\text{else}.
\end{cases}
$$
Indeed, relation~\eqref{lchil1 e1} with $y=3/2$, $q=1$, $\chi(n)=1$ for all $n$, and $k-1$ in place of $k$, implies that
$$
\left|\zeta^{(k-1)}(s)+\frac{(-1)^{k}(k-1)!}{(s-1)^{k}}\right|\ll k!(c_5\log V_t)^k.
$$
Together with Lemma~\ref{lchil3}, this yields the estimate
$$
F^{(k-1)}_\chi(s)\ll \left(\frac{c_6k\log(qV_t)}{\delta(\chi)+(1-\delta(\chi))|L_{qV_t}(s,\chi)|}\right)^k.
$$
Since $V_t\asymp1$ for $|t|\le1/\log(3q)$, this reduces~\eqref{pf e1} to showing that
\be\label{pf e2}
\delta(\chi)+(1-\delta(\chi))|L_{qV_t}(s,\chi)|\asymp
\begin{cases}
1 & \text{if}\ |t|\ge1/\log(3q),\cr
M(\chi) &\text{otherwise}.
\end{cases}
\ee
If, now, $|t|\ge1/\log(3q)\ge1/(3\log(qV_t))$ or $\chi$ is complex, then $|L_{qV_t}(s,\chi)|\asymp1$ by Lemma~\ref{lchil2}, so~\eqref{pf e2} follows.
Also, if $|t|\le1/\log(3q)$ and $\chi$ is principal, that is to say, $\delta(\chi)=1$, then we have trivially that $\delta(\chi)+(1-\delta(\chi))|L_{qV_t}(s,\chi)|=1=M(\chi)$, so~\eqref{pf e2} holds in this case too. Finally, if $|t|\le1/\log(3q)\le1$ and $\chi$ is real and non-principal, then $1.5\le V_0\le V_t\le V_1<3$. In particular, $|t|\le1/\log(qV_t)$, and thus Lemma~\ref{lchil2} implies that
$$
\delta(\chi)+(1-\delta(\chi))|L_{qV_t}(s,\chi)|=|L_{qV_t}(s,\chi)|\gg L_{qV_t}(1,\chi).
$$
Since $V_t\in[1.5,3]$, a continuity argument implies that
\be\label{pf e10}\bsp
L_{qV_t}(1,\chi) =\lim_{\sigma\to1^+}\left\{ L_q(\sigma,\chi) \prod_{q<p\le qV_t} \left( 1-\frac{\chi(p)}{p^\sigma} \right) \right\}
&= L_q(1,\chi) \prod_{q<p\le qV_t} \left( 1-\frac{\chi(p)}{p} \right)\\
& \gg L_q(1,\chi),
\end{split}\ee
which shows~\eqref{pf e2} in this last case too. This completes the proof of~\eqref{pf e1}.

Next, for every integer $k\ge3$ and for every real number $y\ge2$, we apply relation~\eqref{pf e1} to get that
\be\label{proofe1}\bsp
\sum_{n\le y}\Lambda_\chi(n)(\log n)^{k-1}\log\frac yn
&=\frac{(-1)^{k-1}}{2\pi i} \int_{\Re(s)=1+\frac1{\log y}} F_\chi^{(k-1)}(s) \frac{y^s}{s^2}  ds\\
&\ll y \int_{|t|\ge\frac1{\log(3q)}} \frac{ (c_4k\log(qV_t))^k }{t^2+1}dt + \frac{y}{\log(3q)}\left(\frac{c_4k\log(3q)}{M(\chi)} \right)^{k}.
\end{split}\ee
Moreover, we have that
\be\label{pf e3}\bsp
\int_{|t|\ge\frac1{\log(3q)}} \frac{ (c_4k\log(qV_t))^k }{t^2+1} dt & \le\int_{\SR} \frac{(2c_4k\log q)^k+(2c_4k\log V_t)^k }{t^2+1}dt\\
&\ll (c_7k\log(3q))^k+(c_7k)^k\int_e^\infty\frac{(\log t)^{2k/3}(\log\log t)^{k/3}}{t^2} dt \\
& = (c_7k\log(3q))^k+(c_7k)^k\int_1^\infty u^{2k/3}(\log u)^{k/3} e^{-u}du.
\end{split}\ee
The function $u\to \log(u^{2k/3}(\log u)^{k/3}e^{-u/2})$ is concave for $u>1$ and its maximum occurs when $u\asymp k$. So
$$
\int_1^\infty  u^{2k/3}(\log u)^{k/3} e^{-u} du\le (c_8k)^{2k/3}(\log(c_8k))^{k/3}\int_1^\infty  e^{-u/2}du
\ll(c_8k)^{2k/3}(\log(c_8k))^{k/3}
$$
for some $c_8>1$. The above estimate, together with relations~\eqref{proofe1} and~\eqref{pf e3}, implies that
\be\label{pf e4}\begin{split}
\sum_{n\le y}\Lambda_\chi(n)(\log n)^{k-1}\log\frac yn &\ll  y (c_9k^5\log k)^{k/3}+ y (c_9k\log q)^k + \frac{y}{\log(3q)} \left(\frac{c_9k\log(3q)}{M(\chi)} \right)^{k} \\
& \le y (c_9k^{5/3}(\log k)^{1/3})^k+ y \left(\frac{c_{10}k\log(3q)}{M(\chi)} \right)^{k}
\end{split}\ee
for some $c_{10}\ge c_9\ge1$, since $M(\chi)\ll1$, trivially if $\chi$ is complex or principal and by relation~\eqref{lchil1 e2} with $k=0$ and $\epsilon=1/2$ otherwise.

Now, set
$$
\Delta(x) = x\sqrt{\log x}\left\{ \left( \frac{c_9k^{5/3}(\log k)^{1/3}} {\log x} \right)^{k/2}
+ \left( \frac{c_{10}k\log(3q)} {M(\chi)\log x} \right)^{k/2} \right\}
$$
and note that $\Delta(x)\ge\sqrt{x}$, since $c_9k^{5/3}(\log k)^{1/3}\ge k^{5/3}(\log 3)^{1/3}>k\ge x^{-1/k}\log x$. We claim that
\be\label{proofe2a}
\sum_{n\le x}\Lambda_\chi(n)(\log n)^{k-1}\ll\Delta(x)(\log x)^{k-1}\quad(x\ge4).
\ee
If $\Delta(x)>x/2$, then~\eqref{proofe2a} holds trivially. So assume that $\Delta(x)<x/2$. Applying~\eqref{pf e4} for $y=x$ and $y=x-\Delta(x)$ and subtracting one inequality from the other completes the proof of~\eqref{proofe2a}. Relation \eqref{proofe2a} and partial summation imply that
\bes\bsp
\sum_{n\le x}\Lambda_\chi(n)&=O(\sqrt{x})+\int_{\sqrt{x}}^x\frac1{(\log t)^{k-1}}d\left(\sum_{n\le t}\Lambda_\chi(n)(\log n)^{k-1}\right) \ll 2^k\Delta(x)\quad(x\ge16).
\end{split}\ees
If $(\log x)M(\chi)/\log(3q) \ge 8c_{10}$ and $x$ is large enough, then choosing
\[
k= \min \left\{ \frac{ (\log x)^{3/5} }{ 2c_9(\log\log x)^{1/5} }, \frac{(\log x) M(\chi) }{ 2c_{10} \log(3q) } \right\}
\]
yields the estimate
\be\label{pf e20}
\sum_{n\le x}\Lambda_\chi(n)\ll xe^{-c_1(\log x)^{3/5}(\log\log x)^{-1/5}} + x^{1-c_2 M(\chi)/\log(3q)} \sqrt{\log x} .
\ee
This bound holds trivially when $(\log x)M(\chi)/\log(3q)<8c_{10}$ or when $x$ is small as well. Inserting~\eqref{pf e20} into~\eqref{proofe4} completes the proof of the theorem, since there are at most $2\tau_2(q)$ real characters modulo $q$ (for every $d|q$, there at most 2 real primitive characters modulo $q$ of conductor $d$~\cite[Chapter 5]{Dav}).
\end{proof}

As a conclusion, perhaps it is worth noticing that, arguing along the lines of this paper and using some ideas going back to Hal\'asz\footnote{The factor $\sqrt{\log x}$ can be removed by the argument leading to Theorem 1.7 in~\cite{kou}. Finally, Theorem 2.1 in~\cite{kou} and Lemma~\ref{triangle} above can be used to show that, with at most one exception, we have that $L_q(1,\chi)\gg1$ for all real, non-principal characters $\chi$ mod $q$. This allows us to remove the factor $\tau_2(q)$ too.}, it is possible to prove a stronger form of Theorem~\ref{pntap2} with $x^{1-c_4/\log q}+ x^{1-c_5\cdot \eta(q)}/\phi(q)$ in place of $ \sqrt{\log x}\,(x^{1-c_2/\log q} + \tau_2(q) x^{1-c_3\cdot \eta(q)}/\phi(q) )$. Furthermore, the size of $\eta(q)$ can be related to the location of a potential {\it Siegel zero} \cite[Theorem 2.4]{kou}. In particular, it is possible to show that $\eta(q)\asymp1/\log q$ if there are no Siegel zeroes for any characters modulo $q$. So using these elementary methods, we obtain a result which is analogous to the current state of the art on the counting function of primes in an arithmetic progression, originally achieved using the machinery of Complex Analysis. However, showing these improved results is considerably more technical and we have chosen not to include them in this paper in order to keep the presentation simpler. We refer the interested reader to~\cite{kou} for more details.

%%%%%%%%%%%%%%%%%%%%%%%%%%%%%%%%%%%%%%%%%%%%%%%%%%%%%%%%%%%%%%%%%%%%%%%%%%%%%%%%%%%%%%%%%%%%%%%%%%%%%%%%%%%%%%%%%%%%%%%%%%%%
%%%%%%%%%%%%%%%%%%%%%%%%%%%%%%%%%%%%%%%%%%%%%%%%%%%%%%%%%%%%%%%%%%%%%%%%%%%%%%%%%%%%%%%%%%%%%%%%%%%%%%%%%%%%%%%%%%%%%%%%%%%%
%%%%%%%%%%%%%%%%%%%%%%%%%%%%%%%%%%%%%%%%%%%%%%%%%%%%%%%%%%%%%%%%%%%%%%%%%%%%%%%%%%%%%%%%%%%%%%%%%%%%%%%%%%%%%%%%%%%%%%%%%%%%
%%%%%%%%%%%%%%%%%%%%%%%%%%%%%%%%%%%%%%%%%%%%%%%%%%%%%%%%%%%%%%%%%%%%%%%%%%%%%%%%%%%%%%%%%%%%%%%%%%%%%%%%%%%%%%%%%%%%%%%%%%%%
%%%%%%%%%%%%%%%%%%%%%%%%%%%%%%%%%%%%%%%%%%%%%%%%%%%%%%%%%%%%%%%%%%%%%%%%%%%%%%%%%%%%%%%%%%%%%%%%%%%%%%%%%%%%%%%%%%%%%%%%%%%%

\section*{Acknowledgements}

I would like to thank Andrew Granville for many fruitful conversations on the subject of multiplicative functions, and for several comments and suggestions, particularly, for drawing my attention to the proof of the prime number theorem given in~\cite[pp. 40-42]{ik}. Also, I would like to thank Kevin Ford for pointing out~\cite[Corollary 2A]{Fo} to me, and G\'erald Tenenbaum for some helpful comments. Furthermore, I would like to acknowledge the careful reading of the paper by the anonymous referee, who made several helpful suggestions for improving the exposition, exposed some inaccuracies, and pointed out Tatuzawa's paper~\cite{Ta}.

Following the release of a preprint of an earlier version of this paper, Andrew Granville and Eric Naslund independently communicated to me proofs of Lemma~\ref{derlemma} using a generating series argument that also improved the previous version of Lemma~\ref{derlemma} to the current one (the constant 2 in the numerator was 8 before). The proof of Lemma~\ref{derlemma} given here was inspired by these communications and I would like to thank both of them.

This paper was written while I was a postdoctoral fellow at the Centre de recherches math\'ematiques in Montr\'eal, which I would like to thank for the financial support.

%%%%%%%%%%%%%%%%%%%%%%%%%%%%%%%%%%%%%%%%%%%%%%%%%%%%%%%%%%%%%%%%%%%%%%%%%%%%%%%%%%%%%%%%%%%%%%%%%%%%%%%%%%%%%%%%%%%%%%%%%%%%
%%%%%%%%%%%%%%%%%%%%%%%%%%%%%%%%%%%%%%%%%%%%%%%%%%%%%%%%%%%%%%%%%%%%%%%%%%%%%%%%%%%%%%%%%%%%%%%%%%%%%%%%%%%%%%%%%%%%%%%%%%%%
%%%%%%%%%%%%%%%%%%%%%%%%%%%%%%%%%%%%%%%%%%%%%%%%%%%%%%%%%%%%%%%%%%%%%%%%%%%%%%%%%%%%%%%%%%%%%%%%%%%%%%%%%%%%%%%%%%%%%%%%%%%%
%%%%%%%%%%%%%%%%%%%%%%%%%%%%%%%%%%%%%%%%%%%%%%%%%%%%%%%%%%%%%%%%%%%%%%%%%%%%%%%%%%%%%%%%%%%%%%%%%%%%%%%%%%%%%%%%%%%%%%%%%%%%
%%%%%%%%%%%%%%%%%%%%%%%%%%%%%%%%%%%%%%%%%%%%%%%%%%%%%%%%%%%%%%%%%%%%%%%%%%%%%%%%%%%%%%%%%%%%%%%%%%%%%%%%%%%%%%%%%%%%%%%%%%%%

\bibliographystyle{PTW02}

\end{document}